\documentclass[lettersize,journal,onecolumn]{IEEEtran}
\usepackage{amsmath,amsfonts,amsthm, amssymb, mathtools}
\usepackage{algorithmic}
\usepackage{algorithm}
\usepackage{array}
\usepackage[caption=false,font=normalsize,labelfont=sf,textfont=sf]{subfig}
\usepackage{textcomp}
\usepackage{stfloats}
\usepackage{url}
\usepackage{verbatim}
\usepackage{centernot}
\usepackage{graphicx}
\usepackage{cite}
\usepackage[numbers]{natbib}
\usepackage{hyperref}

\newtheorem{theorem}{Theorem}
\newtheorem{lemma}{Lemma}

\newtheorem{remark}{Remark}

\newtheorem{prop}{Proposition}

\def\PROB {{\mathbb P}}
\def\EXP {{\mathbb E}}
\def\IND{{\mathbb I}}
\def\Var{\operatorname{Var}}

\def\P{{\cal P}}
\def\Q{{\cal Q}}
\def\T{{\cal T}}

\def\D{{\cal D}}

\def\R{{\mathbb R}}

\def\bX{{\mathbf{X}}}
\def\bx{{\mathbf{x}}}

\newcommand{\set}[1]{\left\{ #1 \right\}}

\usepackage{fancyvrb} 

\newcommand{\KL}{\operatorname{KL}}

\usepackage{xcolor}

\begin{document}

\title{Tree density estimation}

\author{
\IEEEauthorblockN{L\'aszl\'o Gy\"orfi\IEEEauthorrefmark{1}, Aryeh Kontorovich\IEEEauthorrefmark{2}, 
Roi Weiss\IEEEauthorrefmark{3}}\\
    \IEEEauthorblockA{\IEEEauthorrefmark{1}Department of Computer Science and Information Theory,
Budapest University of Technology and Economics,
Budapest, Hungary}\\
    \IEEEauthorblockA{\IEEEauthorrefmark{2}Department of Computer Science, Ben-Gurion University of the Negev, Beer-Sheva, Israel}\\
    \IEEEauthorblockA{\IEEEauthorrefmark{3}Department of Computer Science,
       Ariel University, Shomron, Israel}
}

\maketitle

\begin{abstract}
We study the problem of 
estimating the density
$f(\bx)$
of
a random vector ${\bX}$ in $\R^d$.
For a spanning tree $T$ defined on the vertex set $\{1,\dots ,d\}$, the tree density $f_{T}$ is a product of bivariate conditional densities.
An {\em optimal} spanning tree
minimizes
the Kullback-Leibler divergence between $f$ and $f_{T}$.
From i.i.d.\ data we 
identify
an optimal tree  $T^*$ and
efficiently construct
a tree density estimate $f_n$ such that, without any regularity conditions on the density $f$, one has 
$\lim_{n\to \infty} \int |f_n(\bx)-f_{T^*}(\bx)|d\bx=0$ a.s. 
For Lipschitz 
$f$ with bounded support,
$\EXP\left\{ \int |f_n(\bx)-f_{T^*}(\bx)|d\bx\right\}=O\big(n^{-1/4}\big)$,
a dimension-free rate.
\end{abstract}

\begin{IEEEkeywords}
density estimation;
tree identification;
Kruskal's algorithm;
consistency;
rate of convergence
\end{IEEEkeywords}

\section{Introduction}
\label{Intr}

A natural strategy for mitigating the curse of dimensionality in 
estimating
probability distributions
is to employ 
low-complexity family of approximation distributions.
For discrete distributions,
\citet{ChLi68} suggested a family of {\em tree-based} approximations and gave an efficient 
maximum-likelihood estimator
based on 
Kruskal's 
optimal spanning tree
algorithm \citep{kruskal56}.
We stress that this approach makes no
{\em structural assumptions} about the sampling distribution, but rather constitutes a {\em modeling choice}.
Consequently, in this paradigm, 
the goal is to approximate the optimal-tree
distribution
from the data, without any guarantees on how well the latter approximates the true sampling distribution.



Extensions 
of the Chow-Liu approach
to continuous
distributions were studied by 
\citet{BaJo03} 
and by 
\citet{LiXuGuLaWa11}
under various assumptions.
They contructed approximation tree densities, which are products of
 bivariate conditional marginals.
A principal contribution of this paper
is to
introduce identification and estimation procedures
for which
strong
consistency
can be established
without {\em any} regularity conditions on the
underlying sampling density.
Our second contribution is to obtain risk rates
under  mild assumptions.

By way of a motivating example,
\citet{HoKoMoNo20}
examine a model
of sensor networks, consisting of $d\gg1$ sensors.
The goal here is
to estimate an approximation to the $d$-dimensional density of the sensors' measurements.
In the special case of
a linear sensor layout,
the joint distribution is naturally
modeled via
a Markov chain.
In this case,
the density is a product of bivariate conditional densities.
In general, we do not assume that the true density has this special form, 
and instead seek
the best approximating product of bivariate conditional densities.
Thus,
out of the possible $\binom{d}{2}$  bivariate conditional densities, an approximation of the true density depends only on $d-1$  bivariate conditional densities.
Furthermore, any approximation has bounded complexity:
it is given by a spanning tree.


\subsection*{Formal setup.}

Let ${\bX}=(X_{1},\dots ,X_{d})\in\R^d$ be a $d$-dimensional random vector with probability density $f(\bx)$.
We denote the index set by
$V=\{1,\dots ,d\}$,
the densities of the one-
and two-dimensional marginals by 
$f_{i}(x_{i})$ and 
$f_{i,j}(x_{i},x_{j})$, 
respectively,
for
$i,j\in V$, $i\neq j$.
A {\em spanning tree}
$T$ on the vertex set $V$
is an undirected graph 
$T=(V,E)$, where $|E|=d-1$;
we denote the set of all spanning trees
by $\T$.
Given the one- and two-dimensional marginals, 
a 
$T\in\T$
induces a joint density on
$\bX$ as follows:
%
\begin{align}
\label{eq:span_tree}
f_{T}(\bx)=\prod\limits_{(i,j)\in T}\frac{f_{i,j}(x_{i},x_{j})}{f_{i}(x_{i})f_{j}(x_{j})}\prod_{i=1}^{d}f_{i}(x_{i}).
\end{align}
Any vertex $k\in V$
may be arbitrarily designated as the {\em root} of $T$.
Once the root is fixed,
any
$i\in V\setminus\set{k}$
has a unique {\em parent} $j(i)=j_T(i)\in V$
(the reader is referred to
\citet{cormen-alg-09}
for the relevant graph-theoretic background,
including spanning trees and Kruskal's algorithm).
The parent notation enables expressing
(\ref{eq:span_tree})
more compactly:
\begin{align}
\label{root}
f_{T}(\bx)=\prod\limits_{i\neq k}f_{i|j(i)}(x_{i}\mid
x_{j(i)})f_k(x_{k}),
\end{align}
where
$k$ is the designated root and
the conditional densities are given by
\begin{align*}
f_{i|j(i)}(x_{i}\mid
x_{j(i)}) = \frac{f_{i,j(i)}(x_{i},x_{j(i)})}{f_{j(i)}(x_{j(i)})}.
\end{align*}

As in 
\citet{DeGy85},  we consider the $L_1$ 
risk
\begin{align}
\label{L1}
\|f-f_n\|=\int |f(\bx) - f_n(\bx) | d\bx;
\end{align}
this choice of metric is amply motivated therein.

This paper investigates
the problem
of finding
a spanning tree $T$ 
on the 
vertex set $V$
and 
establishing that
its induced probability density $f_{T}(\bx)$ 
approximates the 
true
density $f(\bx)$ 
optimally or nearly so, under some criterion.
There are several
candidates for a goodness-of-fit measure,
perhaps the most natural of which
is the $L_1$ metric:
\begin{align*}
\T_{L_1}=
%
\left\{\bar T\in\T: \|f-f_{\bar T}\|=\min_
{T\in\T} \|f-f_{T}\|\right\}
\end{align*}
($\T$ being a finite set, the minimum is always achieved).
The shortcoming of the $L_1$ criterion is
that minimizing it over $T\in\T$ appears to
be a computationally hard problem,
with 
no known efficient approximation algorithm.
Therefore, in line with the original Liu-Chow approach,
we adopt
the Kullback-Leibler (KL) divergence
as our goodness-of-fit criterion.
The KL-divergence between $f(\bx)$ and $f_{T}(\bx)$ is defined by
\begin{align*}
\KL(f,f_{T})=\int f(\bx)\log \frac{f(\bx)}{f_{T}(\bx)}d\bx,
\end{align*}
while the set of optimal spanning trees is
\begin{align*}
\T_{\KL}=\left\{T^*\in\T: \KL(f,f_{T^*})=\min_T \KL(f,f_{T})\right\}.
\end{align*}
Our goal is to identify a spanning tree $T^*$ belonging to the set $\T_{\KL}$ and to estimate a best
tree density $f_{T^*}$ from data.

To present the Chow-Liu approach, let
us assume, for the moment, that
the density $f$ is known.
Consider the undirected complete graph defined on the set of vertices $V$.
The edges of the graph are weighted by the mutual information of the two-dimensional marginal probability distributions corresponding to the two vertices connected.
One can check that
\begin{align}
\label{eq:diff}
\KL(f,f_{T})=\sum_{i=1}^{d}H(X_{i})-\sum_{(i,j)\in T}I(X_{i},X_{j})-H(\bX),
\end{align}
where $I(X_{i},X_{j})$ is the mutual information defined by
\begin{align}
\label{eq:mut_inf}
I(X_{i},X_{j})=\int \int f_{i,j}(x_{i},x_{j})\log \frac{f_{i,j}(x_{i},x_{j})}{f_i(x_{i})f_j(x_{j})}dx_{i}dx_{j},
\end{align}
and $H(\cdot )$ denotes the differential entropy of a random variable or a
random vector
\citep{BaJo03},
defined by
$H(\bX)=H(f)=\int f(\bx)\log\frac1{f(\bx)}d\bx$.
Thus, $\KL(f,f_{T})$ is minimal when we take edges from the complete graph
along the spanning tree having the maximum weight, i.e.,
 $T^*\in \T_{\KL}$ such that
\begin{align*}
T^*
\in
\mathop{\rm arg\, max}_{T}\left( \sum_{(i,j)\in T}I(X_{i},X_{j})\right) .
\end{align*}%

An optimal spanning tree $T^*$ can be obtained by applying Kruskal's algorithm
\citep{kruskal56}
to the 
weighted
complete graph on $V$
described above.
The latter is a simple
method for constructing an optimum spanning tree in an undirected weighted graph, with time complexity $O(d^2\log d)$.
To choose a tree of maximum total edge weight, we
first index the $d(d - 1)/2$ edges according to decreasing
weights $\{b_i\}$, so that $b_i\ge b_j$ whenever $i < j$. 
We then start by selecting $b_1$ and $b_2$, and add $b_3$ if $b_3$ does not form
a cycle with $b_1$ and $b_2$. We continue to consider
edges of successively higher indices, selecting an edge
whenever it does not form a cycle with the set previously
selected, and rejecting it otherwise. 
In general
$T^*$ is not unique;
it is, however,
when all of the $I(X_{i},X_{j})$ are distinct
for $i\neq j$.
Furthermore, $T^*$ depends only on the ordering of the set of mutual informations, $\{I(X_{i},X_{j}),i\neq j\}$.


In the setting of this paper, the density
$f$ is not known. Instead,
$n$ independent copies of $\bX$ 
(the ``data'')
are drawn:
\begin{align*}
\D_{n}=(\bX^{1},\dots ,\bX^{n})
.
\end{align*}%


\citet{TaAnWi10} studied in detail the case when $\bX$ has a multivariate normal distribution.
In this case $I(X_i,X_j) = -\frac{1}{2} \log(1-\rho_{ij}^2)$ where $-1\leq \rho_{ij}\leq 1$ is the correlation coefficient.
In their paper the density $f$ is not arbitrary, it is a tree density $f_{T}$ and an empirical identification algorithm of this tree $T$ is introduced.
If the covariance matrix for the density $f=f_{T}$ has full rank, then the identification error probability has exponential rate of convergence.
Section 8 in \citet{BaJo03} is on stationary Gaussian time series, for which 
$\rho_{ij}=r_{|i-j|}$ with some $r_1,\dots ,r_{d-1}$.
If $|r_1|>|r_{j}|$ for all $j>1$, then the optimal spanning tree $T^*$ is a chain consisting of the edges $(1,2),(2,3),\dots ,(d-1,d)$.
Interestingly, the correlations $r_{j}$ for $j>1$ don't matter.
For general density, this particular spanning tree $T^*$ appears, when $X_1,\dots ,X_d$ is a Markov chain.

\citet{LiXuGuLaWa11} 
considered a more general problem: namely,
the forest density estimate.
For identifying the best forest and for estimating the corresponding forest density, they 
proposed a kernel-based approach.
If the bivariate and univariate densities are H\"older continuous and they are bounded away from zero (called {\em strong density assumption}), then 
under the additional assumption that 
$I(X_{i},X_{j})$ are finite and distinct for $i\neq j$,
the identification of the best
forest is consistent.
Note that the strong density assumption excludes 
many densities of interest, including
Gaussian densities. 
In addition, \citet{LiXuGuLaWa11} give bounds on the rate of convergence of forest density estimate in terms of KL-divergence.
For example, in case of Lipschitz density 
that satisfies the strong density condition, they show that
the excess KL-risk is of order 
\begin{align*}
O\left(\ln n /n^{1/4}\right),
\end{align*}
see Theorem 9 therein. 

\subsection*{Paper overview and main results.}
The main aim of this paper is to avoid the strong density assumption and the assumption that the $I(X_i,X_j)$, $i\neq j$, are distinct.
We introduce a tree estimate $T_n$ and a corresponding tree-density estimate $f_n$ such that without any regularity condition on the density $f$ one has that
\begin{align*}
\lim_{n\to \infty} \int |f_n(\bx)-f_{T_n}(\bx)|d\bx=0
\end{align*}
a.s. (Theorem \ref{Thm:L1}). 
Furthermore, for Lipschitz continuous $f$ with bounded support,
\begin{align*}
\EXP\left\{ \int |f_n(\bx)-f_{T_n}(\bx)|d\bx\right\}
=O\left(n^{-1/4}\right)
\end{align*}
independently on the dimension $d$, (Theorem \ref{Thm:rate}).

For a best spanning tree $T^*$, we have an approximation error:
\begin{align*}
\|f-f_{T^*}\|.
\end{align*}
Pinsker's inequality implies an upper bound:
\begin{align*}
\|f-f_{T^*}\|^2/2\le \KL(f_{T^*},f),
\end{align*}
therefore
\begin{align*}
\|f-f_{T^*}\|\le \sqrt{2\KL(f_{T^*},f)}.
\end{align*}
Using the formula (\ref{eq:diff}), we can estimate $\KL(f_{T^*},f)$. For the identification step, the term $\sum_{(i,j)\in T^*}I(X_{i},X_{j})$ is already estimated, while we can estimate $\sum_{i=1}^{d}H(X_{i})$ and $H(\bX)$ by Kozachenko-Leonenko algorithm.
If $\bar T$ and $T^*$ stand for $L_1$-optimal tree and for KL-optimal tree, respectively, then the previous argument implies a bound on the excess approximation error: 
\begin{align}
\label{KLL1}
0
&\le
\|f-f_{T^*}\|-\|f-f_{\bar T}\|
\le 
\sqrt{2\KL(f_{T^*},f)}-\|f-f_{\bar T}\|
\le 
\sqrt{2\KL(f_{\bar T},f)}-\|f-f_{\bar T}\|.
\end{align}
For perfect approximation, we have $\|f-f_{\bar T}\|=0$, which yields $\KL(f_{\bar T},f)=0$, and 
so (\ref{KLL1}) implies $\|f-f_{T^*}\|=0$, too. 

It is important to characterize the distribution of $\bX$, where
\begin{align*}
\KL(f,f_{T^*})=\min_T \KL(f,f_{T})
\end{align*}
is small. For example, if there is a permutation of the components of $\bX$ such that in this ordering the components form a first order Markov process, then $\KL(f_{T^*},f)=0$.
In general, if $\KL(f_{T^*},f)=0$, then without any regularity condition on the underlying density, our algorithm identifies such perfect spanning tree $T^*$.

An important application of this setup is the example of sensor network, where the sensors are geographically (arbitrary, squared lattice, hexagonal lattice, etc.) distributed. 
We can assume that the mutual information of the neighboring sensors are dominating, and therefore one has to estimate only $c\cdot d$ mutual information $I(X_i,X_j)$ (instead of $\binom{d}{2}$), where $c\approx 4$.
In the Gaussian case mentioned above the correlation matrix has only $c\cdot d$ non negligible elements.
In this sense the correlation matrix is sparse.

\section{The identification of the best spanning tree $T^*$}

\label{Mut}

Using the data $\D_{n}$, we shall construct estimates $I_{n}(X_{i},X_{j})$. Based on these estimates, introduce the empirically
best spanning tree $T_n$:
\begin{align}
\label{Tn}
T_{n}=\mathop{\rm arg\, max}_{T}\left( \sum_{(i,j)\in
T}I_{n}(X_{i},X_{j})\right) .
\end{align}

In order to have universally consistent identification of a best tree, we need consistent mutual information estimates without any regularity
assumption on the underlying density.

Let $(X,Y)$ be a random vector taking values in $\R^{2}$ with
probability density function $f_{XY}(x,y)$ and with marginal densities
$g_{X}(x)$, $g_{Y}(y)$.
The aim is to estimate the mutual information
\begin{align*}
I(X,Y)=\int \int f_{X,Y}(x,y)\log \frac{f_{X,Y}(x,y)}{g_{X}(x)g_{Y}(y)}dxdy
\end{align*}
such that the estimate is strongly consistent without any regularity
condition on the density $f_{X,Y}$.
Assume having the i.i.d.\ data $(X_{1},Y_{1}),\dots ,(X_{n},Y_{n})$.
Several estimates of mutual information have been considered. Most of them are based on density estimates,
from which the consistency of differential entropies estimates is derived.

\citet{GyvdM87} considered histogram-based consistent estimators for differential entropy, which in turn can be used to estimate the mutual information.
The histogram based estimate of mutual information is defined as follows:
Let $\P_{n}$ and $\Q_{n}$ be finite or infinite partitions
of ${\R}$, and denote by $\mu _{n}$ the empirical distribution of  $(X_{1},Y_{1}),\dots ,(X_{n},Y_{n})$:
\begin{align*}
\mu _{n}(A\times B)=\frac 1n \sum_{i=1}^n\IND_{\{X_{i}\in A,Y_{i}\in B\}},
\end{align*}
Set
\begin{align}
\label{est}
I_{n}(X,Y)
=I_{n}(\mu_n,\mu_{n,1}\times \mu_{n,2})
=\sum_{A\in \mathcal{P}_{n},B\in \mathcal{Q}_{n}}
\mu_{n}(A\times B)\log \frac{\mu _{n}(A\times B)}{\mu _{n,1}(A)\mu _{n,2}(B)},
\end{align}
where
\begin{align*}
\mu _{n,1}(A)=\mu _{n}(A\times {\mathbb{R}})
\end{align*}
and
\begin{align*}
\mu _{n,2}(B)=\mu _{n}({\mathbb{R}}\times B).
\end{align*}
\citet*{BaGyvdM92} showed the following: if $\P_{n}=\Q_{n}$ is the uniform partition with bin width $h'_n\to 0$, $nh'^2_n\to \infty$ and $I(X,Y)<\infty$, then
\begin{align}
\label{vdM}
\lim_{n\rightarrow \infty }I_{n}(X,Y)=I(X,Y)
\end{align}
a.s.
\citet*{WaKuVe05} and 
\citet{SiNa10}
introduced and studied data-driven, partitioning-based estimate of the mutual
information.

Let us now return to
the problem of identifying $T^*\in\T_{\KL}$. 
Construct estimates $I_{n}(X_{i},X_{j})$ as in (\ref{est}) and select $T_{n}$ according
to (\ref{Tn}).
For identifying $T^*$, one has to generate $\binom{d}{2}$ mutual information estimate, therefore we have to use a mutual information estimate of small computational complexity.
The estimate (\ref{est}) has the smallest computational complexity among the algorithms mentioned before.
The error of the tree density estimate has two components: error of the identification and the error of the bivariate density estimates. It will turn out that the second error dominates the first one.

\section{Estimating the best  tree density}

\label{Dens}

In this section we study the estimation problem of a best approximating density $f_{T^*}$.
The aim is to introduce a density estimate $f_n$ such that
\begin{align*}
\lim_n\int |f_{T^*}(\bx)-f_n(\bx)|d\bx=0
\end{align*}
a.s. without any regularity conditions on the density $f$ of $\bX$.

One may estimate the original density $f$, for example by the ordinary histogram rule with bin width $h_n>0$,
for which the consistency conditions are $nh_n^d\to\infty$ and $h_n\to 0$.
In applications where $d$ is large,
we typically do not have a sufficiently large sample, i.e., $nh_n^d$ is not sufficiently large.
This is the main motivation why the estimation of the best approximating density $f_{T^*}$ is considered.

For a spanning tree $T_n$, we construct the density estimate $f_n$  by estimating the conditional densities $f_{i\mid j(i)}(x_{i}\mid x_{j(i)})$.
As in  
\citet{GyKo07}, we estimate the conditional densities by the ratio of histograms.
For $1\leq k\leq n$
denote the $k$th sample vector
by $\bX_k=(X^{(k)}_1,\dots ,X^{(k)}_d)$
and
let $\mu_{n,i,j(i)}$ and $\mu_{n,j(i)}$ be the empirical distributions for the samples $((X^{(1)}_i,X^{(1)}_{j(i)}),\dots ,(X^{(n)}_i,X^{(n)}_{j(i)}))$
and $(X^{(1)}_{j(i)},\dots ,X^{(n)}_{j(i)})$, respectively, i.e.,
\begin{align*}
\mu _{n,i,j(i)}(A\times B)=\frac 1n \sum_{k=1}^n\IND_{\{X^{(k)}_i\in A,X^{(k)}_{j(i)}\in B\}}
\end{align*}
and
\begin{align*}
\mu _{n,j(i)}( B)=\frac 1n \sum_{k=1}^n\IND_{\{X^{(k)}_{j(i)}\in B\}}.
\end{align*}

To simplify the analysis,
we renumber the vertex set $V=\{ 1,\dots ,d\}$ such that for any $1\le i<d$, the vertex subset $\{ i,\dots ,d\}$ corresponds to a subtree of $T_n$ with $i$ being a leaf
and its parent satisfies $j(i)>i$.
In particular, $d$ is the root of the tree and the vertices are ordered by their distance from the root.
Let $\P_n$ denote uniform partitions of $\R$ with bin width $h_n$.
For $x_i\in A\in \P_n$, $x_{j(i)}\in B\in \P_n$ and $i=1,\dots ,d-1$,
put
\begin{align*}
f_{n,i\mid j(i)}(x_{i}\mid x_{j(i)})
&=\frac{\mu_{n,i,j(i)}(A\times B)}{h_n\mu_{n,j(i)}(B)}
\end{align*}
with $0/0=0$ by definition and for $x_{d}\in A\in \P_n$,
\begin{align*}
f_{n,d}(x_{d})
&=\frac{\mu_{n,{d}}(A)}{h_n},
\end{align*}
where $d$ is the root of the spanning tree $T_n$.
Set
\begin{align}
\label{histo}
f_n(\bx)
&=
\prod_{i< d}f_{n,i\mid j(i)}(x_{i}\mid x_{j(i)})f_{n,d}(x_{d}).
\end{align}

\section{Consistency and rate of convergence}

Our first result is the density-free strong consistency:

\begin{theorem}
\label{Thm:L1}
Assume that all $I(X_{i},X_{j})$ are finite for $i\neq j$.
If  $h'_n\to 0$, $nh'^2_n\to \infty$, $h_n\to 0$ and $nh_n^2/\log n\to\infty$,  then
\begin{align}
\label{Tas}
\lim_{n\rightarrow \infty }\IND_{T_{n}\in \T_{\KL}}=1
\end{align}
a.s., i.e., 
almost surely,
\begin{align*}
T_{n}\in \T_{\KL}
\end{align*}
for all sufficiently large sample size $n$.
Furthermore,
\begin{align*}
\lim_{n\to \infty} \int |f_n(\bx)-f_{T_n}(\bx)|d\bx=0
\end{align*}
a.s.
\end{theorem}

Notice that the dimension $d$ does not appear in any of the conditions above, which allows for a dimension-free consistency.

The consistency result of Theorem \ref{Thm:L1} holds without any regularity conditions on the density $f$.
Without such conditions, the rate at which the $L_1$ error converges to zero can be arbitrarily slow
\citep{DeGy85}.
So in order to obtain non-trivial rates of convergence, one needs
to impose some regularity condition on $f$.

We say that a function $g:\R^k\to\R$ satisfies the \emph{Lipschitz condition} 
with respect to the Euclidean norm $\|\cdot \|$
if 
for all $\bx,\bx'\in \R^k$,
$$
|g(\bx)-g(\bx')| \leq L \|\bx-\bx'\|.
$$

\begin{theorem}
\label{Thm:rate}
%
If  $\bX$ is bounded,
the one-dimensional marginals  $\{f_{i}\}$ and  the two-dimensional marginals  $\{f_{i,j}\}$
 satisfy the Lipschitz condition, and $nh'^2_n\to \infty$, then
\begin{align}
\label{rate}
\EXP\left\{ \int |f_n(\bx)-f_{T_n}(\bx)|d\bx\right\}
=O\left(1/\sqrt{nh_n^2}\right)+O(h_n)+O({h'}^{\gamma}_n)+O\left(\frac{1}{n{h'}_n^2}\right)
\end{align}
with all $\gamma>0$.
\end{theorem}

\begin{remark}
The rate of convergence in the theorem is dimension-free, which means that it does not depend on $d$.
Note that the constants hidden in the $O$-notation  of the last two terms in \eqref{rate} depend on the density $f$ beyond the Lipschitz constant (i.e., on $\delta>0$ of \eqref{eq:delta} --- 
the minimum gap in the distinct mutual information values of different $X_i,X_j$). 
Thus, the bound is not in the minimax sense. 
\end{remark}

For the choice
\begin{align*}
h_n=c_1\cdot n^{-1/4}
\end{align*}
and
\begin{align*}
h'_n=c_2\cdot n^{-1/4} \quad \mbox{and} \quad \gamma=2,
\end{align*}
$nh'^2_n\to \infty$ and so (\ref{rate}) has the form
\begin{align*}
\EXP\left\{ \int |f_n(\bx)-f_{T_n}(\bx)|d\bx\right\}
=O\left(n^{-1/4}\right)+O\left(n^{-1/2}\right)
=O\left(n^{-1/4}\right)
\end{align*}
such that the density estimation error $O\left(n^{-1/4}\right)$ dominates the identification error $O\left(n^{-1/2}\right)$.
The upper bound $O\left(n^{-1/4}\right)$ on the rate of convergence cannot be improved.
For $d=2$, where there is no tree identification problem, this rate is a minimax lower bound for Lipschitz class and the ordinary histogram achieves this rate, see Lemma \ref{histrate} below. 
A simple embedding argument entails that this also holds for $d>2$ when $\KL(f,f_{T^*})=0$.

\begin{remark}
Recall that \citet{LiXuGuLaWa11} give bounds in a similar setup in terms of the KL-divergence.
In the case of Lipschitz density that satisfies the strong density condition, they show that
the excess KL-risk is of order $O\left(\ln n /n^{1/4}\right)$.
In Theorem \ref{Thm:rate} above we consider the $L^1$ loss. The strong density condition is avoided due to two technical ingredients: (a) the analysis of the approximation error $J_{n,1}$ in the proof of Proposition \ref{pTrate} and (b) the inclusion of the marginal density term in the loss of the conditional density estimation (i.e., \eqref{eq:A}).
\end{remark}

\section{Proof of Theorem  \ref{Thm:L1}}

Because of
\begin{align*}
\int |f_{T_n}(\bx)-f_n(\bx)|d\bx
&\le
2\IND_{T_{n}\notin \T_{\KL}}+ \sum_{ T^*\in \T_{\KL}}\IND_{T_n=T^*} \int |f_{T^*}(\bx)-f_n(\bx)|d\bx,
\end{align*}
the proof of Theorem  \ref{Thm:L1} is decomposed into two propositions.

\begin{prop}
\label{pTcons} Assume that all $I(X_{i},X_{j})$ are finite for $i\neq j$. If $nh'^2_n\to\infty$ and $h'_n\to 0$, then (\ref{Tas}) holds.
\end{prop}

\begin{proof}
The event $\{T_{n}\notin \T_{\KL}\}$ means that the orderings of $\{I_n(X_{i},X_{j}), i\ne j\}$ and of $\{I(X_{i},X_{j}), i\ne j\}$ are distinct.
Put
\begin{align}
\label{eq:delta}
\delta
&=\min_{(i,j)\ne (u,v),|I(X_{i},X_{j})-I(X_{u},X_{v})|>0}|I(X_{i},X_{j})-I(X_{u},X_{v})|.
\end{align}
Then,
\begin{align*}
\{T_{n}\notin \T_{\KL}\}
&\subset 
\cup_{(i,j)\ne (u,v),I(X_{i},X_{j})-I(X_{u},X_{v})> 0 }\{ I_n(X_{i},X_{j})- I_n(X_{u},X_{v})< 0\}\\
&\subset
\cup_{(i,j)}\{| I_n(X_{i},X_{j})-I(X_{i},X_{j})|\ge \delta/2\}.
\end{align*}
Under the conditions of the proposition, (\ref{vdM}) implies that
\begin{align}
\label{as}
\lim_{n\to \infty} I_n( X_{i},X_{j})=I( X_{i},X_{j})
\end{align}
a.s., for all $i\ne j$,
from which the proposition follows.
\end{proof}

\begin{prop}
\label{pfcons}
If  $nh_n^2/\log n\to\infty$ and $h_n\to 0$, then
\begin{align}
\label{cons}
\lim_{n\to \infty} \IND_{T_{n}= T^* }\int |f_n(\bx)-f_{T^*}(\bx)|d\bx=0
\end{align}
a.s.
\end{prop}
\begin{proof}
In the proof of this proposition we apply the strong pointwise consistency of the ordinary histogram.
Choose a sequence of partitions $\P_n$ of $\R^d$ such that the cells of $\P_n$ are shifted versions of $[0,h_n]^d$ with bin width $h_n$, $n=1,2,\dots $.
Assume i.i.d. data
\begin{align*}
\D_n=(\bX_1,\dots ,\bX_n)
\end{align*}
and let $\mu_n$ denote the empirical distribution for $\D_n$.
If $A_n(\bx)$ is the cell of $\P_n$ into which $\bx$ falls, then the histogram estimate is defined by
\begin{align*}
f_{n,h_n}(\bx)=\frac{\mu_n(A_n(\bx))}{h_n^d}.
\end{align*}

\begin{lemma}
\label{hist}
If $h_n\to 0$ and $nh_n^d/\log n\to\infty$, then
\begin{align}
\label{conss}
\lim_{n\to \infty} f_{n,h_n}(\bx)=f(\bx)
\end{align}
a.s. for $\lambda$-almost all $\bx$, where $\lambda$ is the Lebesgue measure.
\end{lemma}
\begin{proof}
Set
\begin{align*}
\bar f_{h_n}(\bx)
&=
\EXP\left\{\frac{\mu_n(A_n(\bx))}{h_n^d} \right\}
=\frac{\mu(A_n(\bx))}{h_n^d}.
\end{align*}
The consistency of the bias term $\bar f_{h_n}(\bx)- f(\bx) $ follows from the generalized Lebesgue density theorem 
\citep[Theorem 7.16]{WhZy77}:
Let $B_n(\bx)$ denote the smallest cube centered at $\bx$ and containing $A_n(\bx)$. If there is a constant $c>0$ such that
\begin{align}
\label{c}
\lambda(B_n(\bx))\le c \lambda(A_n(\bx)),
\end{align}
then $h_n\to 0$ implies that
\begin{align*}
\bar f_{h_n}(\bx)
&=
\frac{\mu(A_n(\bx))}{\lambda(A_n(\bx))}
\to f(\bx)
\end{align*}
for $\lambda$-almost all $\bx$.
Obviously, (\ref{c}) is satisfied with $c=2^d$.
For $\varepsilon>0$, Bernstein's inequality implies
\begin{align*}
\PROB\{|f_{n,h_n}(\bx)-\bar f_{h_n}(\bx)|>\varepsilon\}
&=
\PROB\{|\mu_n(A_n(\bx))-\mu(A_n(\bx))|>\varepsilon h_n^d\}\\
&\le
2e^{-\frac{n\varepsilon^2 h_n^{2d}}{2\mu(A_n(\bx))+2\varepsilon h_n^d/3 }}\\
&=
2e^{-\frac{n\varepsilon^2 h_n^{d}}{2\bar f_{h_n}(\bx)+2\varepsilon /3 }}.
\end{align*}
For $nh_n^d/\log n\to\infty$, this yields
\begin{align*}
\sum_{n=1}^{\infty}\PROB\{|f_{n,h_n}(\bx)-\bar f_{h_n}(\bx)|>\varepsilon\}
&<\infty
\end{align*}
if the sequence $\bar f_{h_n}(\bx)$ is bounded,
and by referring to the Borel-Cantelli lemma the proof of the variance term is complete, i.e.,
\begin{align*}
|f_{n,h_n}(\bx)-\bar f_{h_n}(\bx)|
&\to 0
\end{align*}
a.s.
\end{proof}
We 
now complete
the proof of Proposition \ref{pfcons}. 
First,
we claim
that $f_{T^*}$ is a density, i.e.,
\begin{align}
\label{eq:f_dens}
\int f_{T^*}(\bx)d\bx
&=1.
\end{align}
Indeed, the representation (\ref{root}) implies that
\begin{align*}
\int f_{T^*}(\bx)d\bx
&=
\int\dots \int f_{T^*}(x_1,\dots ,x_d)dx_1\dots dx_d\\
&=
\int\dots \int \prod_{i=1}^{d-1}f_{i\mid j(i)}(x_{i}\mid x_{j(i)})f_d(x_d)dx_1\dots dx_d.
\end{align*}
For the vertex set $\{ 1,\dots ,d\}$, $1$ is a leaf, therefore $j(1)>1$. Thus,
\begin{align*}
\int f_{T^*}(\bx)d\bx
&=
\int\dots \int \left( \int f_{1\mid j(1)}(x_{1}\mid x_{j(1)})dx_1  \right)\prod_{i=2}^{d-1}f_{i\mid j(i)}(x_{i}\mid x_{j(i)})f_d(x_d)dx_2\dots dx_d\\
&=
\int\dots \int \prod_{i=2}^{d-1}f_{i\mid j(i)}(x_{i}\mid x_{j(i)})f_d(x_d)dx_2\dots dx_d.
\end{align*}
\eqref{eq:f_dens} follows
by induction.
Similarly,
 one can check that $f_n$ is also a density, i.e.,
\begin{align*}
\int f_n(\bx)d\bx
&=1.
\end{align*}
For $x_i\in A\in \P_n$, $x_{j(i)}\in B\in \P_n$ and $i=1,\dots ,d-1$, put
\begin{align*}
g_n(x_{i}, x_{j(i)})
&=\frac{\mu_{n,i,j(i)}(A\times B)}{h_n^2}
\end{align*}
and
\begin{align*}
f_{n,j(i)}(x_{j(i)})
&=\frac{\mu_{n,d}(B)}{h_n}.
\end{align*}
Put
\[
D=\{\bx: 0<f_{T^*}(\bx)<\infty\}.
\]
If $\bx\in D$, then $0<f_{i\mid j(i)}(x_{i}\mid x_{j(i)})<\infty$ and so $0<f_{j(i)}(x_{j(i)})$.
Then, Lemma \ref{hist} implies
\begin{align*}
f_{n,i\mid j(i)}(x_{i}\mid x_{j(i)})
&=
\frac{g_n(x_{i}, x_{j(i)})}{f_{n,j(i)}(x_{j(i)})}
\to
f_{i\mid j(i)}(x_{i}\mid x_{j(i)})
\end{align*}
a.s. for $\lambda$-almost all $\bx\in D$.
Thus,
\begin{align}
\label{consss}
f_n(\bx)
&\to
f_{T^*}(\bx)
\end{align}
a.s. for $\lambda$-almost all $\bx\in D$.
The proof is completed by referring to the fact that pointwise consistency implies $L_1$ consistency:
\begin{align*}
\int |f_n(\bx)-f_{T^*}(\bx)|d\bx
&=
2\int (f_{T^*}(\bx)-f_n(\bx))_+d\bx
+\int f_n(\bx)d\bx
-\int f_{T^*}(\bx)d\bx\\
&=
2\int (f_{T^*}(\bx)-f_n(\bx))_+d\bx\\
&=
2\int_D (f_{T^*}(\bx)-f_n(\bx))_+d\bx\\
&\to 0
\end{align*}
a.s., where we used (\ref{consss}) and the dominated convergence theorem.
\end{proof}

\section{Proof of Theorem \ref{Thm:rate}}

Again, the proof of Theorem  \ref{Thm:rate} is decomposed into two propositions.

\begin{prop}
\label{pTrate} 
If  $\bX$ is bounded, the one-dimensional marginals  $\{f_{i}\}$ and  the two-dimensional marginals  $\{f_{i,j}\}$
 satisfy  the Lipschitz condition, and $n{h'}_n^2\to \infty$, then
\begin{align*}
\PROB\{T_{n}\notin \T_{\KL}\}
&\le
O( {h'}^{\gamma}_n)+O\left(\frac{1}{n{h'}_n^2}\right)
\end{align*}
with all $\gamma>0$.
\end{prop}
\begin{proof}
With the notation of Proposition \ref{pTcons}, 
\begin{align*}
\PROB\{T_{n}\notin \T_{\KL}\}
&\le
\max_{(i,j)\ne (u,v)}\IND_{I(X_{i},X_{j})-I(X_{u},X_{v})>0 }\PROB\{ I_n(X_{i},X_{j})- I_n(X_{u},X_{v})< 0\}\\
&\le
\sum_{(i,j)}\PROB\{| I_n(X_{i},X_{j})-I(X_{i},X_{j})|\ge \delta/2\}.
\end{align*}
Therefore, we have to bound the rate of convergence of
\begin{align*}
\PROB\{|I_{n}(X,Y)-I(X,Y)|\ge 5\varepsilon \}
\end{align*}
where $X=X_i$ and $Y=X_j$ with $i\neq j \in \{1,\dots,d\}$ and $\varepsilon=\delta/10$.
We show, that under the condition $n{h'}_n^2\to \infty$,
\begin{align}
\label{hhn}
\PROB\{|I_{n}(X,Y)-I(X,Y)|\ge 5\varepsilon \}
&=O\left({h'}^{\gamma}_n\right)+3e^{-n(\varepsilon +o(1))}+O\left(\frac{1}{n{h'}_n^2}\right).
\end{align}
Consider the decomposition
\begin{align*}
I(X,Y)-I_{n}(X,Y)
&=I(\mu,\mu_{1}\times \mu_{2})-I_{n}(\mu_n,\mu_{n,1}\times \mu_{n,2})
=J_{n,1}+J_{n,2}+J_{n,3},
\end{align*}
where
\begin{align*}
J_{n,1}&=I(\mu,\mu_{1}\times \mu_{2})-I_{n}(\mu,\mu_{1}\times \mu_{2})\\
J_{n,2}&=I_{n}(\mu,\mu_{1}\times \mu_{2})-I_{n}(\mu_n,\mu_{1}\times \mu_{2})\\
J_{n,3}&=I_{n}(\mu_n,\mu_{1}\times \mu_{2})-I_{n}(\mu_n,\mu_{n,1}\times \mu_{n,2})
=I_{n}(\mu_{n,1},\mu_{1})+I_{n}(\mu_{n,2},\mu_{2}).
\end{align*}
To bound the approximation term $J_{n,1}$, note that
\begin{align*}
J_{n,1}\ge 0,
\end{align*}
because the mutual information is a KL-divergence, and $ I_{n}(\mu,\mu_{1}\times \mu_{2})$ is a KL-divergence restricted to a product of partitions. Thus, instead of looking at $|J_{n,1}|$, it is enough to upper bound $J_{n,1}$.
Let $\bar f_{X,Y}(x,y) =\frac{1}{h_n^2}\int\int_{A_n(x,y)} f_{X,Y}(x',y') dx' dy' =  \frac{\mu(A_n(x,y))}{h_n^2}$, where 
$A_n(x,y)$ is the cell of the product partition into which $(x,y)$ falls, and similarly $\bar g_X$ and $\bar g_Y$ for the marginals $g_X$ and $g_Y$.
We have that
\begin{align*}
    J_{n,1} &= \int\int f_{X,Y}(x,y)\log{\frac{f_{X,Y}(x,y)}{g_X(x) g_Y(y)}} - \int\int\bar f_{X,Y}(x,y)\log{\frac{\bar f_{X,Y}(x,y)}{\bar g_X(x) \bar g_Y(y)}}
    \\ &
    = 
    \int\int f_{X,Y}(x,y)\log{\frac{f_{X,Y}(x,y)}{g_X(x) g_Y(y)}} dxdy
    - \int\int f_{X,Y}(x,y)\log{\frac{\bar f_{X,Y}(x,y)}{\bar g_X(x) \bar g_Y(y)}}dxdy
    \\&\quad
    + \int\int f_{X,Y}(x,y)\log{\frac{\bar f_{X,Y}(x,y)}{\bar g_X(x) \bar g_Y(y)}}dxdy
    - \int\int\bar f_{X,Y}(x,y)\log{\frac{\bar f_{X,Y}(x,y)}{\bar g_X(x) \bar g_Y(y)}}dxdy.
\end{align*}
The second line in the last equation is zero while the first one is
\begin{align*}
    D_{KL}(f_{X,Y},\bar f_{X,Y}) - D_{KL}(g_X,\bar g_X) - D_{KL}(g_Y,\bar g_Y) \leq D_{KL}(f_{X,Y},\bar f_{X,Y}).
\end{align*}
Since $J_{n,1}\geq 0$, it is left to show that $D_{KL}(f_{X,Y},\bar f_{X,Y})=O({h'_n})$.
To this end,
applying the Lipschitz condition one gets that
\begin{align*}
D_{KL}(f_{X,Y},\bar f_{ X,Y})
\le 
D_{\chi^2}(f_{X,Y},\bar f_{X,Y})
&=
\int \frac{(f_{X,Y}(x ,y)-\bar f_{X,Y}(x,y))^2}{\bar f_{X,Y}(x,y)}dx{dy}
\\ &
\le 
Lh_n'\int \frac{|f_{X,Y}(x,y)-\bar f_{X,Y}(x,y)|}{\bar f_{X,Y}(x,y)}dxdy.
\end{align*}
Therefore,
\begin{align*}
D_{KL}(f_{X,Y},\bar f_{X,Y})
&\le 
Lh_n'\int \frac{f_{X,Y}(x,y)+\bar f_{X,Y}(x,y)}{\bar f_{X,Y}(x,y)}dx dy
=
Lh_n'\int \frac{\bar f_{X,Y}(x,y)+\bar f_{X,Y}(x,y)}{\bar f_{X,Y}(x,y)}dx dy
  = 
2CLh_n',
\end{align*}
where $C$ is the Lebesgue measure of the support of $\bar f$.
Thus,
\begin{align*}
\IND_{J_{n,1}\ge \varepsilon }
&\le J_{n,1}^{\gamma}/\varepsilon^{\gamma}
\le 2^{\gamma}C^{\gamma}L^{\gamma}{h'}^{\gamma}_n /\varepsilon^{\gamma}.
\end{align*}
For KL-divergence restricted to finite partitions, 
\citet{Tus77,Kal85,QuRo85,Bar89}
proved exponential, large deviation-type inequalities, 
\citep[Section 3.1]{GrGy10}.
From \citet[Equation (13)]{GrGy10}, the boundedness of $X$ and $Y$ and $nh'_n\to \infty$ yield
\begin{align*}
\PROB\{I_{n}(\mu_{n,1},\mu_{1})>\varepsilon\}
&=e^{-n(\varepsilon +o(1))}
\end{align*}
and
\begin{align*}
\PROB\{I_{n}(\mu_{n,2},\mu_{2})>\varepsilon\}
&=e^{-n(\varepsilon +o(1))}.
\end{align*}
We have that
\begin{align*}
J_{n,2}
&=I_{n}(\mu,\mu_{1}\times \mu_{2})-I_{n}(\mu_n,\mu)+I_{n}(\mu_n,\mu)-I_{n}(\mu_n,\mu_{1}\times \mu_{2})\\
&=J_{n,4}-I_{n}(\mu_n,\mu),
\end{align*}
where
\begin{align*}
J_{n,4}
=\sum_{A\in \mathcal{P}_{n},B\in \mathcal{Q}_{n}}
(\mu(A\times B)-\mu_{n}(A\times B))\log \frac{\mu(A\times B)}{\mu _{1}(A)\mu _{2}(B)},
\end{align*}
Again, the boundedness of $X$ and $Y$ and $n{h'}_n^2\to \infty$ yield
\begin{align*}
\PROB\{I_{n}(\mu_{n},\mu)>\varepsilon\}
&=e^{-n(\varepsilon +o(1))}.
\end{align*}
The Cauchy-Schwarz inequality implies that
\begin{align*}
\Var(J_{n,4})
&\le
|\mathcal{P}_{n}|\cdot |\mathcal{Q}_{n}|\sum_{A\in \mathcal{P}_{n},B\in \mathcal{Q}_{n}}\Var\left(\mu_{n}(A\times B)\log \frac{\mu(A\times B)}{\mu _{1}(A)\mu _{2}(B)} \right)\\
&\le
\frac{|\mathcal{P}_{n}|\cdot |\mathcal{Q}_{n}|}{n}  \sum_{A\in \mathcal{P}_{n},B\in \mathcal{Q}_{n}}
\mu(A\times B)\left(\log \frac{\mu(A\times B)}{\mu _{1}(A)\mu _{2}(B)}   \right)^2.
\end{align*}
Therefore, Chebyshev's inequality implies
\begin{align*}
\PROB\{|J_{n,4}|\ge \varepsilon \}
&\le \Var(J_{n,4})/\varepsilon^2
=O\left(\frac{1}{nh_n'^2}\right),
\end{align*}
where we 
used the fact
that under the conditions of the proposition
\begin{align*}
\sum_{A\in \mathcal{P}_{n},B\in \mathcal{Q}_{n}}\mu(A\times B)\left(\log \frac{\mu(A\times B)}{\mu _{1}(A)\mu _{2}(B)}   \right)^2
&\to
\int \int f_{X,Y}(x,y)\left(\log \frac{f_{X,Y}(x,y)}{g_{X}(x)g_{Y}(y)} \right)^2 dxdy<\infty .
\end{align*}
The last inequality is proved in the following lemma.
\end{proof}

\begin{lemma} If  $X,Y$ are bounded 
and the density $f_{X,Y}$ satisfies the Lipschitz condition, then 
\begin{align}
\label{eq:integral_bound}
I(X,Y)^2 \leq    \int \int f_{X,Y}(x,y)\left(\log \frac{f_{X,Y}(x,y)}{g_{X}(x)g_{Y}(y)} \right)^2 dxdy<\infty .
\end{align}
\end{lemma}
\begin{proof}
By Jensen's inequality
\begin{align*}
\int \int f_{X,Y}(x,y)\left(\log \frac{f_{X,Y}(x,y)}{g_{X}(x)g_{Y}(y)} \right)^2 dxdy
&\geq 
\left(\int \int f_{X,Y}(x,y)\log \frac{f_{X,Y}(x,y)}{g_{X}(x)g_{Y}(y)}  dxdy\right)^2
=
I(X,Y)^2.
\end{align*}
To show the second inequality in \eqref{eq:integral_bound},
let $L> 0$ be the Lipschitz constant of $f_{X,Y}$ and note that 
for any $(x,y)\in \R^2$,
\begin{align*}
    g_X(x) & = \int_{-\infty}^\infty f_{X,Y}(x,y') dy' 
    \geq \int_{y-\frac{f_{X,Y}(x,y)}{L}}^{y+\frac{f_{X,Y}(x,y)}{L}} f_{X,Y}(x,y') dy' 
    \geq \int_{-\frac{f_{X,Y}(x,y)}{L}}^{\frac{f_{X,Y}(x,y)}{L}} (f_{X,Y}(x,y) - L |y'|) dy' 
    = \frac{f_{X,Y}(x,y)^2}{L}.
\end{align*}
Similarly, $g_Y(y)\geq \frac{f_{X,Y}(x,y)^2}{L}$.
%
We write 
\begin{align*}
    \int \int f_{X,Y}(x,y)\left(\log \frac{f_{X,Y}(x,y)}{g_{X}(x)g_{Y}(y)} \right)^2 dxdy
    &= \int \int g_{X}(x)g_{Y}(y)\frac{f_{X,Y}(x,y)}{g_{X}(x)g_{Y}(y)}\left(\log \frac{f_{X,Y}(x,y)}{g_{X}(x)g_{Y}(y)} \right)^2 dxdy
\end{align*}
and split the integral's domain into those 
$x,y$ satisfying $\frac{f_{X,Y}(x,y)}{g_{X}(x)g_{Y}(y)} < 1$ and those for which $\frac{f_{X,Y}(x,y)}{g_{X}(x)g_{Y}(y)} \geq 1$.
Since $t\log^2 t$ is bounded for $0\leq t \leq 1$, the integral over the first domain is bounded. For the second domain, since $\log^2 t$ is monotonic increasing for $t\geq 1$,
we use the bounds $g_X(x)\geq \frac{f_{X,Y}(x,y)^2}{L}$ and $g_Y(y)\geq \frac{f_{X,Y}(x,y)^2}{L}$ above to get
\begin{align*}
    \int \int_{\frac{f_{X,Y}(x,y)}{g_{X}(x)g_{Y}(y)} \geq 1} f_{X,Y}(x,y)\left(\log \frac{f_{X,Y}(x,y)}{g_{X}(x)g_{Y}(y)} \right)^2 dxdy
    &
    \leq 
    \int \int_{\frac{f_{X,Y}(x,y)}{g_{X}(x)g_{Y}(y)} \geq 1} f_{X,Y}(x,y)\left(\log \frac{L^2}{f_{X,Y}(x,y)^3} \right)^2 dxdy
    \\&
    \leq 
    \int \int f_{X,Y}(x,y)\left(\log \frac{L^2}{f_{X,Y}(x,y)^3} \right)^2 dxdy
    \\ &= 
    9 L^{2/3}\int \int \frac{f_{X,Y}(x,y)}{L^{2/3}}\left(\log \frac{f_{X,Y}(x,y)}{L^{2/3}}\right)^2 dxdy
.
\end{align*}
Since $f_{X,Y}$ is Lipschitz in a bounded domain, it is bounded, and since $t\log^2 t$ is bounded when $t\geq 0$ is bounded, the last integral is bounded as well.
\end{proof}

\begin{prop}
\label{pfrate}
If $\bX$ is bounded, the one-dimensional marginals  $\{f_{i}\}$ and  the two-dimensional marginals  $\{f_{i,j}\}$ satisfy the Lipschitz condition, then
\begin{align*}
\EXP\left\{ \IND_{T_{n}= T^* }\int |f_n(\bx)-f_{T^*}(\bx)|d\bx\right\}=O\left(1/\sqrt{nh_n^2}\right)+O(h_n).
\end{align*}
\end{prop}
\begin{proof}
Here we apply the rate of convergence for the histogram estimation rule
given in Lemma~\ref{histrate} below.
Recall that we renumber the vertex set $V=\{ 1,\dots ,d\}$ such that for any $1\le i<d$, the vertex subset $\{ i,\dots ,d\}$ corresponds to a subtree of $T^*$ with $i$ being a leaf and $j(i)>i$, and so $d$ is the root of the subtree and the subtree vertices are ordered by their distance from the root.
For the sake of simplicity we use the abbreviation
\begin{align*}
f(x_{i}\mid x_{j(i)})
&=
f_{i\mid j(i)}(x_{i}\mid x_{j(i)})
\end{align*}
and
\begin{align*}
f_{n}(x_{i}\mid x_{j(i)})
&=
f_{n,i\mid j(i)}(x_{i}\mid x_{j(i)}).
\end{align*}
If a void product is defined to be $1$, then we have the decomposition
\begin{align*}
f_n(\bx)-f_{T^*}(\bx)
&=
\prod_{i=1}^{d-1}f_n(x_{i}\mid x_{j(i)})f_n(x_d)-\prod_{i=1}^{d-1}f(x_{i}\mid x_{j(i)})f(x_d)\\
&=
\sum_{k=1}^{d-1}
\prod_{i=1}^{k-1}f_n(x_{i}\mid x_{j(i)})[f_n(x_{k}\mid x_{j(k)})-f(x_{k}\mid x_{j(k)})]\prod_{i=k+1}^{d-1}f(x_{i}\mid x_{j(i)})f(x_d)\\
&\quad +
\prod_{i=1}^{d-1}f_n(x_{i}\mid x_{j(i)})[f_n(x_d)-f(x_d)].
\end{align*}
Thus,
\begin{align*}
|f_n(\bx)-f_{T^*}(\bx)|
&\le
\sum_{k=1}^{d-1}
\prod_{i=1}^{k-1}f_n(x_{i}\mid x_{j(i)})|f_n(x_{k}\mid x_{j(k)})-f(x_{k}\mid x_{j(k)})|\prod_{i=k+1}^{d-1}f(x_{i}\mid x_{j(i)})f(x_d)\\
&\quad +
\prod_{i=1}^{d-1}f_n(x_{i}\mid x_{j(i)})|f_n(x_d)-f(x_d)|.
\end{align*}
For $k\le d-1$, we have that
\begin{align*}
&\int \dots \int
\prod_{i=1}^{k-1}f_n(x_{i}\mid x_{j(i)})|f_n(x_{k}\mid x_{j(k)})-f(x_{k}\mid x_{j(k)})|\prod_{i=k+1}^{d-1}f(x_{i}\mid x_{j(i)})f(x_d)dx_1\dots dx_d\\
&=
\int \dots \int
|f_n(x_{k}\mid x_{j(k)})-f(x_{k}\mid x_{j(k)})|\prod_{i=k+1}^{d-1}f(x_{i}\mid x_{j(i)})f(x_d)dx_k\dots dx_d,
\end{align*}
while
\begin{align*}
\int \dots \int
\prod_{i=1}^{d-1}f_n(x_{i}\mid x_{j(i)})|f_n(x_d)-f(x_d)|dx_1\dots dx_d
&=
\int |f_n(x_d)-f(x_d)|dx_d.
\end{align*}
This last term is easier, because according to the rate of convergence theorem of the standard histogram the conditions of the theorem imply
\begin{align*}
\EXP \int |f_n(x_d)-f(x_d)|dx_d=O\left(1/\sqrt{nh_n}\right)+O(h_n),
\end{align*}
(cf. Lemma \ref{histrate} below).
The problem left is to show that for $k\le d-1$,
\begin{align*}
& \int \dots \int
\EXP\{|f_n(x_{k}\mid x_{j(k)})-f(x_{k}\mid x_{j(k)})|\}\prod_{i=k+1}^{d-1}f(x_{i}\mid x_{j(i)})f(x_d)dx_k\dots dx_d\\
&=O\left(1/\sqrt{nh_n^2}\right)+O(h_n)
\end{align*}
By the ordering of the variables, $k < j(k)$ and the unique path from the root $d$ up to vertex $j(k)$ contains only vertices from $\{j(k)+1,\dots,d\}$, ordered by their distance from $d$.
Let $d = j_1 > j_2 >\dots > j_r = j(k)$ be this path.
Then,
\begin{align}
\nonumber
& \int \dots \int
\EXP\{|f_n(x_{k}\mid x_{j(k)})-f(x_{k}\mid x_{j(k)})|\}\prod_{i=k+1}^{d-1}f(x_{i}\mid x_{j(i)})f(x_d)dx_k\dots dx_d
\\
\nonumber
&=
\int \dots \int
\EXP\{|f_n(x_{k}\mid x_{j(k)})-f(x_{k}\mid x_{j(k)})|\}
 dx_k
\\
\nonumber
&\qquad\qquad \times
f(x_{j_r} \mid x_{j_{r-1}}) f(x_{j_{r-1}} \mid x_{j_{r-2}}) \dots f(x_{j_{2}} \mid x_{j_1}) f(x_{j_{1}})
dx_{j_1} \dots dx_{j_r}
\\
\nonumber
&=
\int \int
\EXP\{|f_n(x_{k}\mid x_{j(k)})-f(x_{k}\mid x_{j(k)})|\}
f(x_{j(k)})  dx_k dx_{j(k)}
\\
\label{eq:A}
&=
\int
\EXP\left\{\int|f_n(x_{k}\mid x_{j(k)})-f(x_{k}\mid x_{j(k)})|
f(x_{j(k)}) dx_k\right\}
 dx_{j(k)}.
\end{align}
Lemma \ref{lem:L1_bound} below implies that for any $x_{j(k)}$,
\begin{align*}
\int |f_n(x_{k}\mid x_{j(k)})-f(x_{k}\mid x_{j(k)})|f(x_{j(k)}) dx_k
& =
\int
\left|\frac{g_n(x_{k}, x_{j(k)})}{f_n(x_{j(k)})}-\frac{f(x_{k}, x_{j(k)})}{f(x_{j(k)})}\right| f(x_{j(k)}) dx_k
\\
&\leq
2\int \frac{|g_n(x_{k}, x_{j(k)})- f(x_{k}, x_{j(k)})| f(x_{j(k)})}{\max\{f_n(x_{j(k)}),f(x_{j(k)})\}} dx_k
\\
&\leq
2\int |g_n(x_{k}, x_{j(k)})- f(x_{k}, x_{j(k)})| dx_k,
\end{align*}
where we used the fact that
\[
\frac{f(x_{j(k)})}{\max\{f_n(x_{j(k)}),f(x_{j(k)})\}}\leq 1.
\]
Thus,
\eqref{eq:A} is bounded from above by
\begin{align*}
2\int
\EXP\left\{
\int |g_n(x_{k}, x_{j(k)})- f(x_{k}, x_{j(k)})|
\right\}
dx_k dx_{j(k)}
& = O\left(1/\sqrt{nh_n^2}\right)+O(h_n),
\end{align*}
where the last equality is by Lemma \ref{histrate}.
\end{proof}

\begin{lemma}
\label{histrate}
\citep{BeGy98,DeGy85}.
For the notation of Lemma \ref{hist}, assume that $h_n\to 0$ and $nh_n^d\to\infty$. If $\bX$ is bounded and $f$ is Lipschitz continuous, then
\begin{align*}
\EXP\left\{\int |f_{n,h_n}(\bx)-f(\bx)|d\bx\right\}
&=
O\left(1/\sqrt{nh_n^d}\right)+O(h_n).
\end{align*}
\end{lemma}

\begin{lemma}
\label{lem:L1_bound}
Let $g,h:\R\to\R$ be nonnegative
with
$\tilde g = \int g(x)dx < \infty$
and $\tilde h = \int h(x)dx < \infty$.
Then,
\begin{align*}
\int \left\vert \frac{g(x)}{
\tilde g
} - \frac{h(x)}{
\tilde h
} \right\vert dx
\leq
2\int \frac{\left\vert g(x) - h(x)\right\vert}{\max\{\tilde g,\tilde h\}} dx.
\end{align*}
\end{lemma}
\begin{proof}
This was proven for the $\ell_2$
norm in \citet[Lemma 3.1]{LeeGT14},
but the argument works for any norm.
We have
\begin{align*}
\int \left\vert \frac{g(x)}{\tilde g} - \frac{h(x)}{\tilde h} \right\vert dx
&\leq
\int \left\vert \frac{g(x)}{\tilde g} - \frac{g(x)}{\tilde h} \right\vert dx
+\int \left\vert \frac{g(x)}{\tilde h} - \frac{h(x)}{\tilde h} \right\vert dx
\\
&=
\left\vert \frac{1}{\tilde g} - \frac{1}{\tilde h} \right\vert
\tilde g
+
\frac{1}{\tilde h}
\int \left\vert g(x)-h(x) \right\vert dx
\\
&=
 \frac{\left\vert \tilde h-\tilde g\right\vert}{\tilde h}
+
\frac{1}{\tilde h}
\int \left\vert g(x)-h(x) \right\vert dx
\\
&\leq \frac{2}{\tilde h} \int \left\vert g(x)-h(x) \right\vert dx.
\end{align*}
By symmetry,
\begin{align*}
\int \left\vert \frac{g(x)}{\tilde g} - \frac{h(x)}{\tilde h} \right\vert dx
&\leq \frac{2}{\tilde g} \int \left\vert g(x)-h(x) \right\vert dx
\end{align*}
as well. The claim now follows.
\end{proof}

\section{Acknowledgment}
We thank the anonymous referees for a number of suggestions and corrections,
including catching a substantive mistake in an earlier version
and suggesting a way of fixing it.

\bibliographystyle{plainnat}
\bibliography{refs}

\end{document}